\documentclass[11pt]{amsart}

\usepackage{amsmath,amsthm,amssymb,latexsym,amstext,amsfonts,mathabx}
\usepackage[english]{babel}
\usepackage{geometry}
\usepackage{graphics,comment}
\usepackage{graphicx}
\usepackage{enumitem}
\usepackage{bbold}
\usepackage{tikz}
\usepackage{tikz-cd}
\usetikzlibrary{%
  matrix,%
  calc,%
  arrows,%
  shapes,%
  positioning%
}
\usetikzlibrary{positioning}
\usepackage{hyperref}
\usepackage{caption}
\usepackage{subfigure}
\usepackage{algorithmic}
\usepackage{cite}

\usepackage{ifthen,xkeyval,xcolor,calc,graphicx}

\usepackage{amsaddr}

\newtheorem{theorem}{Theorem}[section]

\newtheorem{remark}[theorem]{Remark}

\DeclareMathOperator{\colim}{colim}

\newcommand{\R}{\mathbb{R}}
\newcommand{\A}{\mathbb{A}}

\newcommand{\X}{\mathbb{X}}

\newcommand{\PPP}{\mathbb{P}}

\def \HH{{\mathcal H}}

\def \DD{{\mathcal D}}

\def \Pe{{\mathcal P}}

\def \BB{{\mathcal B}}

\DeclareMathOperator{\im}{im}

\newcommand{\set}[1]{\{\,#1\,\}}
\newcommand{\x}{\mathbb{x}}
\newcommand{\Hg}{\mathrm{H}}
\newcommand{\rto}{\rightarrow}
\newcommand{\kk}{\mathrm{k}}
\newcommand{\meet}{\wedge}
\newcommand{\join}{\vee}

\newcommand{\from}{\leftarrow}
\newcommand{\bigmeet}{\bigwedge}
\newcommand{\bigjoin}{\bigvee}

\newcommand{\authorfootnotes}{\renewcommand\thefootnote{\@fnsymbol\c@footnote}}%

\makeatletter
\let\start@align@nopar\start@align
\let\start@gather@nopar\start@gather
\let\start@multline@nopar\start@multline
\long\def\start@align{\par\start@align@nopar}
\long\def\start@gather{\par\start@gather@nopar}
\long\def\start@multline{\par\start@multline@nopar}
\makeatother

\begin{document}

\title{Aspects of an internal logic for persistence}
\author[JPC, P\v S and MVJ]{%
Jo\~ao Pita Costa,  
Primo\v z  \v Skraba,  
Mikael Vejdemo-Johansson }
\address{Proceedings of the Applied Topology Conference 2013\\ 
Journal of Topological Methods in Nonlinear Analysis}
\date{\today}


\begin{abstract} 
The foundational character of certain algebraic structures as Boolean algebras and Heyting algebras is rooted in their potential to model classical and constructive logic, respectively. In this paper we discuss the contributions of algebraic logic to the study of persistence  based on a new operation on the ordered structure of the input diagram of vector spaces and linear maps given by a filtration. Within the context of persistence theory, we give an analysis of the  underlying algebra, derive universal properties and discuss new applications. We highlight the definition of the implication operation within this construction, as well as interpret its meaning within persistent homology, multidimensional persistence and zig-zag persistence. 
\end{abstract}

\maketitle


\section{Introduction}
\label{Introduction}

Topological data analysis has been a vibrant area of research a lot due to the developments in applied and computational algebraic topology. 
It applies the qualitative methods of topology to problems of machine learning, data mining and computer vision \cite{TD}. 
Persistent homology, as the most widely applied tool from
computational algebraic topology has been applied  to problems in 
machine learning \cite{Ede13},
data mining \cite{Wag12},
robotics \cite{Str13}, 
social media \cite{Wag12}, 
large scale data \cite{Har12}, 
natural image statistics \cite{Car08},
nonlinear systems \cite{Cha13}, 
medicine and cancer research \cite{Nic11},
and development of more accurate models \cite{Bub12}.
In the past years several extensions of persistence were proposed, including zig-zag persistence~(cf. \cite{ZigZag1} and \cite{ZigZag2}), and multidimensional persistence~(cf. \cite{Multi}). 

Recently,  persistent homology has been studied using techniques from lattice theory \cite{Skr13} with several algorithmic applications and structural consequences \cite{SoCG14}. 
This is of particular interest for diagrams of vector spaces of different
shapes.  
With the appropriate definitions, such diagrams form lattices: partially ordered sets (posets) with uniquely determined greatest lower bounds and least upper bounds, encoded as additional binary operations named \emph{meet} and \emph{join}, denoted by $\meet$ and $\join$. 
These lattices are equipped with the structure of a \emph{Heyting algebra}: in addition to meet and join acting as conjunctive and disjunctive operators, they also admit an implication operator allowing them to serve as algebraic models of constructive logic in the same way as Boolean algebras model classical logic (cf. \cite{Skr13}).

Constructive logic replaces the traditional concept of truth with the concept of constructive provability and is associated with a transition from the proof to model theory of abstract truth where semantics mirrors classical Boolean-valued semantics using Heyting algebras in place of Boolean algebras (cf. \cite{Be06}). 
The proofs produced by constructive logic have the existence property,
making it suitable for the algorithmic construction of examples from a constructive proof of the existence of certain object.  
Already in \cite{Markov}, Markov expresses that the significance for mathematics of rendering more precise the concept of algorithm emerges in connection with a certain constructive foundation for mathematics, and the further applications of such work.
Heyting algebras are presented in \ref{Heyting algebras and Boolean algebras} and their associated constructive logic will be briefly discussed in subsection \ref{The logic of Heyting algebras}.
Due to its constructive nature, this logic presents a different perspective, as discussed in Section~\ref{An internal logic for persistence}.

In~\cite{Mik12} the relevance of Heyting algebras for the study of persistent homology is further clarified, in a wider
context of the unification by topos theory.
In this paper we present a further step in this direction of research
towards such topos theoretical foundations. Our main contribution is
the definition and interpretation 
of the implication operation over the underlying Heyting algebra constructed in \cite{Skr13}. 
In particular, we analyze the implication between two vector spaces of a given diagram in the context of standard, multidimensional and zig-zag persistence. 
The latter case is motivated by the lattice construction for zig-zag persistence from \cite{SoCG14}.
Furthermore, we will give an interpretation of the implication operation and discuss aspects of the correspondent internal logic given by the underlying algebra.

We assume that the reader has a basic familiarity with algebraic topological notions such as (co)homology, simplicial complexes, filtrations, etc. For an overview, we recommend the references \cite{Hat00} for algebraic topology, as well as \cite{Zom05} for applied and computational topology.


%
\section{Preliminaries}
\label{Preliminaries}

\subsection{Heyting algebras and Boolean algebras}
\label{Heyting algebras and Boolean algebras}

Partially ordered sets (or posets) are common objects of study in topology. 
A lattice is a poset such that all pairs of elements $x,y$ have a greatest lower bound, denoted by $x\meet y$, and least upper bound, denoted by $x\join y$. 
Lattices are posets with minimal conditions to acquire an algebraic structure given by the binary operations $\meet$ and $\join$ determined by the following axioms:
\begin{itemize}
\item [$(1)$] $x\meet(y\meet z)=(x\meet y)\meet z$  and $x\join (y\join z)=(x\join y)\join z$ (associativity);
\item [$(2)$] $x\meet y=y\meet x$ and $x\join y=y\join x$ (commutativity);
\item [$(3)$] $(y\meet x)\join x=x$ and $(y\join x)\meet x=x$ (absorption).
\end{itemize}
The equivalence between this algebraic perspective of a lattice $L$ and its ordered perspective is given by the following equivalence: for all $x,y\in L$, $x\leq y$ iff $x\meet y=x$ iff $x\join y=y$.

Other axioms may be added to this theory in order to describe other subclasses of lattices. 
The class of  \emph{distributive} lattices is determined by the axioms above together with one of the following equivalent axioms:
\begin{itemize}
\item[$(4a)$] $x\meet (y\join z)=(x\meet y)\join (x\meet z)$; 
\item[$(4b)$] $x\join (y\meet z)=(x\join y)\meet (x\join z)$;
\item[$(4c)$] $(x\join y) \meet (x\join z) \meet (y\join z) = (x\meet y)\join (x\meet z)\join (y\meet z)$. 
\end{itemize}
Examples of lattices include the power set of a set ordered by subset inclusion, or the collection of all partitions of a set ordered by refinement  (where the first is distributive but the second is not).
Hence, the order and the algebraic structures hold the same information over different perspectives. 
A lattice is \emph{complete} If every (possibly infinite) subset of a lattice has a supremum and an infimum.
All finite lattices are complete. 
Every lattice can be determined by a unique undirected graph for which the vertices are the lattice elements and the edges correspond to the partial order: the \emph{Hasse diagram} of the lattice. 
Whenever all elements are order-related, the poset is called a \emph{totally ordered set}. 
Totally ordered sets are always distributive lattices.

A \emph{Boolean algebra} $\BB$ is a distributive lattice with a least element $\perp$ and a greatest element $\top$ where, for all elements $a\in \BB$, there exists $b\in \BB$ such that $(a\join b) = \top$ and  $(a\meet b) = \perp$ ($b$ is called the \emph{complement} of $a$ and denoted by $\neg a$). 
Observe that $a\join \neg b$ is the greatest element $x$ satisfying the inequality $(a\meet x)\leq b$.
Though, if not all elements in a lattice have a complement, i.e. such an $x$ may not exist.
A \emph{Heyting algebra} is a distributive lattice $\HH$ with a least element $\perp$ such that, for all $a,b\in  \HH$ there exists a greatest element $x\in \HH$ satisfying $(a\meet x)\leq b$. 
This element is called the \emph{relative pseudo-complement} of $a$ with respect to $b$ denoted by $a\Rightarrow b$.
With it we define $\neg a$ as $a\Rightarrow \perp$ for all $a\in \HH$.  
The collection of all subsets of a given set, with intersection and union as $\meet$ and $\join$, is a Boolean algebra.
On the other hand, totally ordered sets and Heyting algebras are
examples of distributive lattices that in general are not Boolean algebras.
The Heyting algebras $\HH$ that are Boolean algebras are the ones satisfying $\neg \neg x=x$, for all $x\in \HH$.
Heyting algebras always have a greatest element given by $\perp \Rightarrow \perp$.
Boolean algebras are examples of Heyting algebras. 
Moreover, every complete distributive lattice constitutes a Heyting
algebra with the implication operation given by 
\[
x \Rightarrow y = \bigjoin \set{x\in L \mid (x \meet a) \leq b}.
\]

Any Heyting algebra must satisfy the \emph{infinite distributivity}
identity: 
\[
x\meet \bigjoin_{i\in I}y_{i}=\bigjoin_{i\in I}(x\meet y_{i})
\]

\subsection{The logic of Heyting algebras}
\label{The logic of Heyting algebras} 
 
To understand the role of Heyting algebras on their internal logic, we can compare them to Boolean algebras which model classical logic (cf. \cite{Joh86}).
In classical logic, the truth values that a formula can take are usually chosen as the members of a Boolean algebra. 
A corresponding theorem is true for constructive logic if, instead of assigning each formula a value from a Boolean algebra, one uses values from a Heyting algebra.
Consider a collection $\mathbb{P} = \set{p_1,\dots}$ of truth value variables and all logical propositions that can be built from these.
Given a Heyting algebra $\A = (A; \meet, \join, \Rightarrow, \perp)$ and an assignment $v:\mathbb{P}\to A$ of truth values in the Heyting algebra, we can extend the truth values recursively to the entire set of logical propositions by using the following rules: $v(\varphi \join \psi)=v(\varphi) \join v(\psi)$; $v(\varphi \meet \psi)=v(\varphi) \meet v(\psi)$; $v(\varphi \Rightarrow \psi)=(v(\varphi) \Rightarrow v(\psi))$; and $v(\perp)=\perp$.
A formula $\varphi$ is \emph{valid} in $\A$ under a valuation $v$ if $v(\varphi)=\top$; the formula $\varphi$ is valid in $\A$ if it is valid for every valuation $v$ of $\A$.
The axioms of constructive logic are thus determined by the smallest set of formulas that are valid in every Heyting algebra, while classical logic satisfies exactly the formulas that are valid in every Boolean algebra. 
The axioms of constructive logic may be given by the following list
\begin{itemize}
\item[CL 1.] $p\Rightarrow (q\Rightarrow p)$;
\item[CL 2.] $(p\Rightarrow (q\Rightarrow r))\Rightarrow ((p\Rightarrow q)\Rightarrow (p\Rightarrow r))$;
\item[CL 3.] $(p\meet q) \Rightarrow p$;
\item[CL 4.] $(p\meet q) \Rightarrow q$;
\item[CL 5.] $(p\Rightarrow p)\join q$;
\item[CL 6.] $(q\Rightarrow p)\join q$;
\item[CL 7.] $(p\Rightarrow r)\Rightarrow ((q\Rightarrow r)\Rightarrow ((p\join q)\Rightarrow r))$;
\item[CL 8.] $\perp \Rightarrow p$.
\end{itemize}
Indeed, the set of axioms of classical logic properly contains the set of axioms of constructive logic as classical logic includes $p\join \neg p$ and $(\neg \neg p)\Rightarrow p$ (cf. \cite{Be06}).
To recognise valid formulas, it is sufficient to consider the Heyting algebra of open subsets of the real line $\R$, where the lattice operations $\meet $ and $\join$ correspond to the set intersection and union, $\cap$ and $\cup$.
The value assigned to a formula $a\Rightarrow b$ is the interior of the union of the value of $B$ and the complement of the value of $A$, $int(A^C\cup B)$, the bottom element is the empty set, $\emptyset$ and the top element is the entire line $\R$.
The value of $\neg A$, defined by $A\Rightarrow \emptyset  $, reduces to the interior of the complement of the value of $A$, i.e., the exterior of $A$. 
With this setting the open sets of the real line constitute a Heyting algebra which is not a Boolean algebra, as the complement of an open set need not be open.
The formulas assigned to the value of the entire line $\R$ are exactly the valid constructive formulas.  
Conversely, for every invalid formula, there is an assignment of values to the variables that yields a valuation that differs from the top element $\top$. 
To read in more detail the general connection between logic, topology and Heyting algebras, please read \cite{Vik96}.


\section{A Heyting algebra of vector spaces}
\label{A Heyting algebra of vector spaces}

\subsection{Motivation} 
In \cite{Skr13}, natural lattice operations where  defined for
standard and multidimensional persistence, as well as constructed lattice operations for general diagrams. 
When considering the standard persistent homology, presented in \cite{Ede02}, we take a topological space $\X$ and a real-valued function $f:\X \rto \R$. The object of study of persistent homology is a filtration of a space $\X$, i.e., $\emptyset = \X_0 \subseteq \X_1 \subseteq \X_2 \subseteq \ldots\subseteq\X_{N-1} \subseteq \X_N = \X$.
Assuming that this is a discrete finite filtration of tame spaces, we
take the homology of each of the associated chain complexes and
obtain 
\[
 \Hg_*(\X_0) \rto  \Hg_*(\X_1) \rto  \Hg_*(\X_2) \rto \ldots\rto
 \Hg_*(\X_{N-1}) \rto  \Hg_*(\X_N).
\]
If homology is taken over a field $\kk$, the resulting homology groups are vector spaces and the induced maps are linear maps. 
The lattice operations $\meet$ and $\join$ of the underlying order structure can be defined as follows:
\[
(\Hg_*(\X_i) \join \Hg_*(\X_j)) = \Hg_*(X_{\max(i,j) })
\]
\[
(\Hg_*(\X_i) \meet \Hg_*(\X_j)) = \Hg_*(X_{\min(i,j) })
\]
The persistent homology groups can then be rewritten as follows: for
any two elements $\Hg_*(\X_i)$ and $ \Hg_*(\X_j)$, the rank of the
persistent homology classes is 
\begin{equation}\label{eq:persistence}
\im (\Hg_*(\X_i\meet\X_j) \rto \Hg_*(\X_i\join \X_j)).
\end{equation}
In the context of multidimensional persistence, such lattice operations in a bifiltration can naturally be defined by the following equations:  
\[
\Hg_*(\X_{i,j}) \join \Hg_*(\X_{k,\ell}) =
\Hg_*(X_{\max(i,k),\max(j,\ell) })
\]
\[
\Hg_*(\X_{i,j}) \meet \Hg_*(\X_{k,\ell}) =
\Hg_*(X_{\min(i,k),\min(j,\ell) })
\]

The lattice-theoretic definition of persistent homology groups agrees
with both the standard case and the rank invariant in multidimensional
persistence. 
This definition leads to a generalisation of persistence, enriching the poset of vector spaces and linear maps to which we call \emph{diagram} (of vector spaces and linear maps not necessarily sharing domains and codomains) with an underlying algebraic structure given by two lattice operations $\meet $ and $\join$ with nice properties like associativity or commutativity and an infinite notion of distributivity.

\subsection{Revisiting the lattice construction}
When considering an arbitrary commutative diagram of vector spaces and linear maps, a partial order can be introduced where the elements are those vector spaces, and the linear maps determine the order relations between them. 
To do so we consider a directed acyclic graph of vector spaces $G$ together with respective linear maps, assuming one unique component.
The partial order $\leq $ is thus given by $A\leq B \text{  if there exists a linear map  } f:A\rightarrow B \text{  in the input diagram}.$ 
The ordered structure is a poset correspondent to the linear maps in the commutative diagram of spaces given as input. We consider the equivalence of vector spaces denoted by $A\leftrightsquigarrow B$ if there is an isomorphism between $A$ and $B$, in order to, without loss of generality, identify all isomorphic structures (cf. \cite{Skr13}).
Note that the partial order as given does not yet constitute a lattice.
The construction of the lattice operations $\meet $ and $\join$ described in \cite{Skr13} extends the poset into a complete Heyting algebra. These constructions are based on direct sums and categorical limits and colimits.
To avoid dense notation, the extension of the partial order $\leq$ will be noted by the same symbol, being part of the larger partial order.
The operations are constructed as follows: take arbitrary elements $A$ and $B$ of the input poset. 
The \emph{meet} of $A$ and $B$, $A\meet B$, and the \emph{join} of $A$ and $B$, $A\join B$, correspond to the greatest lower bound of $\set{A,B}$ and the least upper bound of $\set{A,B}$, respectively.
Formally, given arbitrary spaces $X$ and $Y$ in a diagram $\DD$, 
\begin{itemize}
\item[(i)] $X\meet Y$ is the intersection of all pullbacks of common targets to $X$ and $Y$, a subspace of $X\oplus Y$ mapping to $X$ and $Y$ by projection;
\item[(ii)] $X\join Y$ is the quotient of $X\oplus Y$ by the sum of all kernels of projections onto pushouts of common sources to $X$ and $Y$, a quotient of $X\oplus Y$ such that either $X$ or $Y$ maps to this join by mapping through their direct sum.
\end{itemize}
Note that given another vector space $D$ such that $D\leq A,B$, it must be below $A\meet B$ due to its construction as a limit.
Hence, $A\meet B$ is the greatest lower bound and, similarly, $A\join B$ is the least upper bound of the set $\set{A,B}$ due the universality of its construction as a colimit.
Both of these operations extend to finite joins $\bigjoin_i D_i$ and meets $\bigmeet_i D_i$ (that might not be in $\DD$ but in the underlying lattice to which we complete $\DD$).
For a diagram $\DD$ and a collection $\set{D_i}$ of spaces, we have  
  \[
  X\meet Y = \lim\set{X\to Z\from Y : Z \text{ common target of $X$ and $Y$}}
  \]
  and also 
  \[
  X\join Y = \colim\set{X\from Z\to Y : 
    Z \text{ common source of $X$ and $Y$}}
  \]
The operations $\join $ and $\meet $ defined above determine a lattice of vector spaces: the partially ordered set $\Pe=(\PPP;\leq)$, where $\PPP$ is the closure of the input poset $P$ relative to these operations. We refer to it as the \emph{persistence lattice} of a given diagram of vector spaces and linear maps, i.e., the completion of that diagram into a lattice structure using the lattice operations $\join $ and $\meet $ (cf. \cite{Skr13}).
These constructions may be computed using  algorithms described in \cite{Skr13a} and in \cite{SoCG14}.
For \emph{persistence} in a general diagram we use the definition from
Equation~\ref{eq:persistence}: for any two elements $\X_i$ and $\X_j$, the rank of the persistent homology classes is 
\[ \im (\Hg_*(\X_i\meet\X_j) \rto \Hg_*(\X_i\join \X_j)). \]
In \cite{Skr13}, it was shown that persistence lattices are complete: due to the nature of their lattice operations, they can be defined in $X=\oplus_\ell \set{A_{\ell}\in S}$ to an arbitrary family of spaces $\{A_{\ell}\}$ in the input diagram. 
Moreover, they are distributive lattices, thus constitute a complete Heyting algebras.
Therefore,  the following infinite notion of distributivity 
\[
X\meet \bigjoin_{i\in I}Y_{i}=\bigjoin_{i\in I}(X\meet Y_{i}).
\]
is satisfied by the underlying algebra of any input diagram of vector spaces and linear maps.
This identity, known as the infinite distributive law, ensures commutativity of binary meets with infinite joins. 
To analyze the construction of the implication operation consider an arbitrary family of spaces $\{X_{\ell}\}$, and the colimit 
\[
A \Rightarrow B=\bigjoin \set{X_{\ell}\in L \mid \text{  a linear map  }
  \bigoplus_{\ell} (X_{\ell}\meet A) \rightarrow B \text{  exists }}.
\]
This general construction of the implication operation permits us a global perspective, enabling techniques as the algorithm for the greatest injective discussed in \cite{Skr13} and \cite{SoCG14}.

\subsection{An internal logic for persistence} 
\label{An internal logic for persistence} 

For constructive mathematics, the existence of an object is equivalent to the possibility of its construction and, unlike the classical approach, the existence of an entity cannot be proved by refuting its non-existence \cite{Jap07}. 
While in classical logic, the negation of a statement asserts that the statement is false, for constructivism it must be refutable and thus $P$ is a stronger statement then $\neg \neg P$.
In particular, the law of excluded middle, \emph{A {\bf or} not A}, is
not accepted as a universally valid principle, although \emph{A {\bf and} not A} is still not true (cf. \cite{Hey56}).
Hence, constructive mathematics  differs from classical mathematics, the former being more appropriate to computability.  
%
Therefore,  a constructive mathematical framework is more 
computational  in the sense that it provides certificates of
existence in the form of algorithmic constructions. 

\begin{remark}\label{axm}
The list below exhibits several properties satisfied by the vector spaces in the underlying structure of any persistence lattice.
They derive from properties satisfied by any Heyting algebra, due to \cite{Be06}, \cite{Bo94}, \cite{Mac92} and \cite{Hey56}.
\begin{itemize}
\item[(1)] $(A\Rightarrow A)=\top$; 
\item[(2)] $(A\meet (A\Rightarrow B))=(A\meet B)$; 
\item[(3)] $(B\meet (A\Rightarrow B))=B$; 
\item[(4)] $(A\Rightarrow (B\meet C))=((A\Rightarrow B)\meet (A\Rightarrow C))$;  
\item[(5)] $A\leq (B\Rightarrow A)$ and $\neg A \leq (A\Rightarrow B)$;
\item[(6)] $(A\Rightarrow (B\Rightarrow C))\leq ((A\Rightarrow B)\Rightarrow (A\Rightarrow C))$;
\item[(7)] $(A\Rightarrow C)\leq ((B\Rightarrow C)\Rightarrow ((A\join B)\Rightarrow C))$;
\item[(8)] $(A\Rightarrow B) \leq ((A\Rightarrow \neg B) \Rightarrow \neg A)$;
\item[(9)] $((A\meet B)\Rightarrow C)=(A\Rightarrow (B\Rightarrow C))$;
\item[(10)] $(A\Rightarrow (B\Rightarrow (A\join B)))=\top$;
\item[(11)] $(A\Rightarrow (A\join B))=\top$ and $(B\Rightarrow (A\join B))=\top$;
\item[(12)] $((A\Rightarrow C)\Rightarrow ((B\Rightarrow C)\Rightarrow ((A\join B)\Rightarrow C)))=\top$; 
\item[(13)] $(0 \Rightarrow A) =\top$;
\item[(14)] if $(A\Rightarrow B)=\top$ and $(B\Rightarrow A)=\top$ then $A=B$;
\item[(15)] if $(\top\Rightarrow B)=\top$ then $B=\top$;
\item[(16)] $(A\meet \neg A) =0$ and $(A\join \neg A)=\top$;
\item[(17)] if $A\leq B$ then $\neg B\leq \neg A$;
\item[(18)] $A\leq \neg \neg A$ and $\neg A = \neg \neg \neg A$;
\item[(19)] if $A$ has a complement, it must be $\neg A$;
\item[(20)] the lattice of vector spaces $\HH$ is a Boolean algebra iff $\neg \neg A$, for all $A\in \HH$.
\end{itemize}
\end{remark}
Note that (9) and (10) above, are results following directly from the definition of the pseudo-complement. 
Others, like (11), (12) and (13) follow directly from the above (CL 3), (CL 4), (CL 7) and (CL 8) presented in the subsection \ref{The logic of Heyting algebras}.
Moreover, (16) to (20), specify $\neg$ and tell us how close we are from dealing with a Boolean algebras. 

Such universal laws permit the simplification of algebraic expressions involving the implication operation in the framework of persistence, as shown in subsection \ref{Zig-zag case}.    
Surely our achievement is not the proof of these laws, but rather the construction of a Heyting algebra on persistence such that we can eventually use these laws that hold for Heyting algebras in general.
They illustrate a certain logic embodiment internal to this algebra.

\section{Computing aspects of the internal logic}
\label{Computing internal logic aspects}

\subsection{Interpretation of the implication operation}
The main contribution of this paper is to interpret the \emph{implication
operation}  for the persistence lattice.
A Heyting algebra describes the order structure of the collection of open sets of a topological space (cf. \cite{Joh86}). 
In such a model, \textit{modus ponens} is the main property expected from the implication: given $U$ and $U\Rightarrow W$ we may infer $W$, an entailment relation that can be expressed as $(U\meet (U\Rightarrow W))\leq W$. The implication is asked to be the weakest possible such assumption. 
The exponential $U\Rightarrow V$, for arbitrary open sets $U$ and $V$ can be expressed by the union $\bigcup W_i$ of all open sets $W_i$ for which $W_i\cap U\subset V$ as follows: as the intersection is distributive over arbitrary unions, $(\bigcup W_i)\cap U=\bigcup (W_i\cap U)\subset V$ and, thus, $\bigcup W_i = (U\Rightarrow V)$.
In the framework of persistence lattices, given arbitrary vector spaces $A$ and $B$, $A\Rightarrow B$ is the join of all elements of the lattice $C_i$ such that a linear map $f: (C_i\meet A) \rightarrow B$ exists.
Indeed, whenever $A$ and $B$ are vector spaces in a diagram, there exists a vector space $X$ that is maximal in the sense of $(X\meet A)\leq B$, i.e., in the sense of the existence of a map $\chi$ as in the diagram below:

\begin{center}
\begin{tikzpicture}[scale=.5]

  \node (d) at (0,-2.5) {$A\meet B$};
  \node (a) at (-2.5,0) {$A$};
  \node (b) at (5,2.5) {$X=(A\Rightarrow B)$};
  \node (x) at (2.5,0) {$B$};
  
\draw[arrows=-latex'] (d) -- (a) node[pos=.5,left] {$\phi_{A}$};
\draw[arrows=-latex'] (x) -- (b) node[pos=.5,right] {$\chi$};
\draw[arrows=-latex'] (d) -- (x) node[pos=.5,right] {$\phi_{B}$};

\end{tikzpicture}
\end{center}
Observe that $(A\meet X) \leq B$ so that $A\meet X=(A\meet X)\meet B=A\meet (X\meet B)=A\meet B$.   
The following results describe the implication operation in several different cases, within the framework of persistence.

\subsection{Standard persistence case} 

Recall that all complete totally ordered sets are bounded by a largest
element and a smallest element, denoted by $\top$ and $\perp$ respectively. 
In the context of a filtration of topological spaces, $\top$ corresponds
to the final or terminal space while $\perp$ corresponds to the initial
space or the empty set.
We have the following description of the implication operation:

\begin{theorem}\label{std}
Let $\HH$ be the underlying Heyting algebra of a totally ordered filtration and $\X_i\in \HH$ for all $i$ in a set of indexes $I$. 
Then, 
\[
\X_i\Rightarrow \X_j = \begin{cases} \X_j, & \mbox{if } \X_j\leq
  \X_i\\ \top, & \mbox{if } \X_i\leq \X_j \end{cases}
\]
\end{theorem}

\noindent In this case, given vector spaces $A$ and $B$ in the filtration, $A\Rightarrow B$ gives us the total space in the filtration, the biggest element $\top$ of the respective lattice, when a linear map $A\rightarrow B$ exists in the input diagram; or else $B$ whenever a linear map $B\rightarrow A$ exists in the input diagram.  
Hence, in such a totally ordered set the pseudo-complement of $A$ is always zero, i.e., $\neg A=(A\Rightarrow \perp)=\perp$ .

\subsection{Multidimensional persistence case}
For the multidimensional persistence case we focus on a bifiltration,
that is, a filtration in two parameters.  
We assume that in general the commutative squares are bicartesian, so that the pushouts and pullbacks are the edges of the original diagram.
In this case, the
description of the implication operation when two elements are related is
similar to the description above for the standard persistence case.   

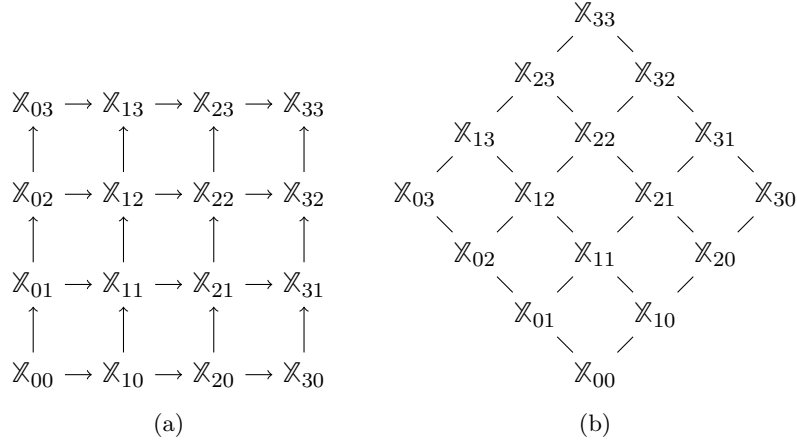
\begin{figure}
\centering
\subfigure[]{%
     \begin{tikzpicture}[scale=.6] 
         \foreach \i  [evaluate=\i as \x using int(2*\i)] in {0,...,3}{
           \foreach \j [evaluate=\j as \y using int(2*\j)] in {0,...,3}{
             \node (p\i\j) at (\x,\y) {$\X_{\i\j}$}    ;
           }
         }
         \foreach \i [evaluate=\i as \x using int(\i+1)] in {0,...,2}{
           \foreach \j [evaluate=\j as \y using int(\j+1)]  in {0,...,2}{
             \draw[->] (p\i\j) -- (p\i\y)  ;
             \draw[->] (p\i\j) -- (p\x\j)  ;
           }
         }
         \foreach \j [evaluate=\j as \y using int(\j+1)]  in {0,...,2}{
           \draw[->] (p3\j) -- (p3\y)  ;
           \draw[->] (p\j3) -- (p\y3)  ;
           }
     \end{tikzpicture}
}     
 \quad
 \subfigure[]{%
     \begin{tikzpicture}[scale=.8]

  \node (03) at (-3,1){$\X_{03}$} ;
  \node (13) at (-2,2){$\X_{13}$} ;
  \node (23) at (-1,3){$\X_{23}$} ;
  \node (33) at (0,4){$\X_{33}$} ;
  \node (32) at (1,3){$\X_{32}$} ;  
  \node (31) at (2,2){$\X_{31}$} ;
  \node (30) at (3,1){$\X_{30}$} ;  
  \node (22) at (0,2){$\X_{22}$} ;
  \node (12) at (-1,1) {$\X_{12}$} ; 
  \node (21) at (1,1){$\X_{21}$} ;
  \node (02) at (-2,0){$\X_{02}$} ;
  \node (20) at (2,0){$\X_{20}$} ;
  \node (11) at (0,0) {$\X_{11}$} ;
  \node (01) at (-1,-1){$\X_{01}$} ;
  \node (10) at (1,-1) {$\X_{10}$} ;
  \node (00) at (0,-2) {$\X_{00}$} ;
  \draw (00) -- (01) -- (02) -- (03) -- (13) -- (23) -- (33) -- (32) -- (31)-- (30) -- (20) -- (10) -- (00);
  \draw (02) -- (12) -- (22) -- (32);
  \draw (01) -- (11) -- (21) -- (31);
  \draw (10) -- (11) -- (12) -- (13);      
  \draw (20) -- (21) -- (22) -- (23);  
\end{tikzpicture}
}
\caption{The diagram of a bifiltration of dimensions $4\times 4$ (a) and the Hasse diagram of the correspondent underlying Heyting algebra (b).}
\label{figmultidim}       
\end{figure}     
Consider the Hasse diagram of the underlying algebra correspondent to a bifiltration of dimensions $4\times 4$ in Figure \ref{figmultidim}.
In that diagram, $\X_{01}\leq \X_{31}$ and thus $(\X_{01}\Rightarrow \X_{31})=\X_{33}=\top$ while $(\X_{31}\Rightarrow \X_{01})=\X_{01}$ 
For unrelated elements in the diagram above, as $\X_{02}$ and $\X_{11}$ for instance, we get $(\X_{02}\Rightarrow \X_{11}) = \X_{31}$ and $(\X_{11}\Rightarrow \X_{02})=\X_{03}$.
Moreover, the pseudo-complement has nontrivial behavior: clearly $(\X_{03}\Rightarrow 0)=\X_{30}$  as well as $(\X_{30}\Rightarrow \perp)=\X_{03}$ but $(\X_{20}\Rightarrow 0)=\X_{03}$ also. 
On the other hand, $(\X_{11}\Rightarrow \perp)=\perp$.
\begin{figure}
\centering
\begin{tikzpicture}[scale=0.6]

    \node (0) at (0,-3) {$\X_{00}$} ;
    \node (nm) at (0,3) {$\X_{mn}$} ;
    \node (0n) at (-3,0) {$\X_{0n}$} ;
    \node (m0) at (3,0) {$\X_{m0}$} ;

    \node (1) at (-1,-2) {} ;
    \node (2) at (1,-2) {} ;
    \node (3) at (-2,-1) {} ;
    \node (4) at (0,-1) {} ;
    \node (5) at (2,-1) {} ;
    \node (6) at (-1,0) {} ;
    \node (7) at (1,0) {} ;
    \node (8) at (-2,1) {} ;
    \node (9) at (0,1) {} ;
    \node (10) at (2,1) {} ;
    \node (11) at (-1,2) {} ;
    \node (12) at (1,2) {} ;
        
    \draw (0) -- (1) -- (3) -- (0n) -- (8) -- (11) -- (nm) -- (12) -- (10) -- (m0) -- (5) -- (2) -- (0);
    \draw (12) -- (9) -- (6) -- (3) -- (1) -- (4) -- (7) -- (10) -- (12) ;
    \draw (8) -- (6) -- (4) -- (2) -- (5) -- (7) -- (9) -- (11) -- (8) ;
       
\end{tikzpicture}
\caption{The Hasse diagram of the underlying Heyting algebra of a bifiltration of dimensions $m\times n$.}
\label{figmultidimx}       
\end{figure}
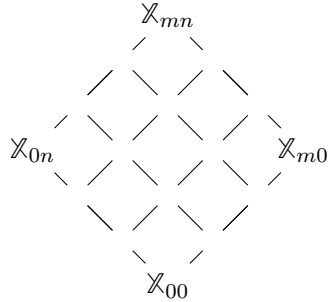     

In general, consider a bifiltration with dimensions $n\times m$ with the corresponding underlying algebra represented by the Hasse diagram of Figure \ref{figmultidimx}.
Let $A=\X_{xy}$ and $B=\X_{zw}$ be arbitrary related vector spaces in the given persistence lattice. 
Notice that, if $A\leq B$ or $B\leq A$, $\top$ is the biggest element of the correspondent totally ordered set where $A$ and $B$ belong, that in the case of the underlying algebra of such a bifiltration is $\X_{mn}$.  
Note also that $x\leq z$ and $y\leq w$ together imply that $A\leq B$.
Similarly, the inequality $A\geq B$ is implied by $x\geq z$ and $y\geq w$ together. 
Considering this, the following result can express the implication operation for such a general diagram.

\begin{theorem}\label{bif}
Let $\HH$ be the underlying Heyting algebra of a bifiltration and let $\X_{ij}$ be arbitrary unrelated vector spaces in the given persistence lattice. 
Assume without loss of generality that $x\leq z$ and $y\geq w$.
Then, 
\[
 \X_{xy} \Rightarrow \X_{zw} = \begin{cases}  \top=\X_{mn}, & \mbox{if
   } \X_{xy}\leq \X_{zw} \\ \X_{zw}, & \mbox{if   } \X_{zw}\leq
   \X_{xy}\\  \X_{mw} & \text{otherwise.  } \end{cases}.
\]
\end{theorem}

\begin{proof}
If $A=\X_{xy}$ and $B=\X_{zw}$ are such that $A\leq B$ or $B\leq A$, we fall into the case of totally ordered sets similar to the case of standard persistence, giving us $\top$ or $B$, respectively.
On the other hand, we can determine the implication operation for unrelated elements in the context of this bifiltration by the biggest element $\X$ such that $(A\meet \X) \leq B$, the element of maximal $i$ (or $j$ if $x\geq z$) right above $B$: that is $\X_{mw}$.
\end{proof}

\begin{remark}
The implication operation gives us the element of maximal uncertainty with respect to a direction as in the standard case. 
Due to its generality, an analogous description holds in higher order filtrations as well as in arbitrary diagrams (although in arbitrary diagrams it does not necessarily have such a nice form).
\end{remark}

A natural decomposition of a vector space of the underlying lattice structure $\HH$ is given by the lattice operations in the following sense: when given $A\in \HH$ one can say that $A=(B\join C)$ for some $B,C\in \HH$.
A similar decomposition can be described for the meet operation $\meet$. An element $A\in \HH$ is called \emph{join-irreducible} if, for all $B,C\in \HH$ such that $A=(B\join C)$, we get $B=A$ or $C=A$. 
Observe that the join-irreducible elements of the persistence lattice of the bifiltration of dimension $2\times 2$ above are $\perp$, $A$, $D$, $F$ and $E$.
For all the others, a decomposition as a join of distinct elements is possible: for instance $G=(A\join B)$.
In general, the join-irreducible elements of a persistence lattice of a bifiltration of dimension $m\times n$ are the elements of the totally ordered sets $\X_{00}\rightarrow \dots \rightarrow \X_{0n}$ and $\X_{00}\rightarrow \dots \rightarrow \X_{m0}$.
Hence, the elements that admit a nonzero pseudo-complement coincide with the join-irreducible elements as described in the next result.
It exhibits a deeper relation between irreducibility and pseudo-complements in Heyting algebras. 
\begin{theorem}
The only elements having nonzero pseudo-complements are the elements of the totally ordered sets $\X_{00}\rightarrow \dots \rightarrow \X_{0n}$ and $\X_{00}\rightarrow \dots \rightarrow \X_{m0}$ thus coinciding with the join-irreducible elements.
\end{theorem}

\subsection{Zig-zag persistence case}
\label{Zig-zag case}

Using the constructions in \cite{SoCG14}, we give a definition of the implication operation in the case of zig-zag persistence.
The interpretation of the implication operation in this case brings us to the analysis of a diagram representing an underlying algebra similar to the one achieved in the multidimensional case, considering bifiltrations.
A zig-zag module is \emph{normalized} if the arrows alternate in direction. Any zig-zag can be transformed in a normalized zig-zag module simply by introducing copies of modules and identity maps in appropriate directions.
Consider the normalized zig-zag module given in \cite{Sk12} by $\X_{0} \rightarrow \X_{01}\leftarrow \X_{1}\rightarrow \X_{12}\leftarrow \X_{2}\rightarrow \X_{23}\leftarrow \X_{3}$ with the diagram of Figure \ref{figzig}. 

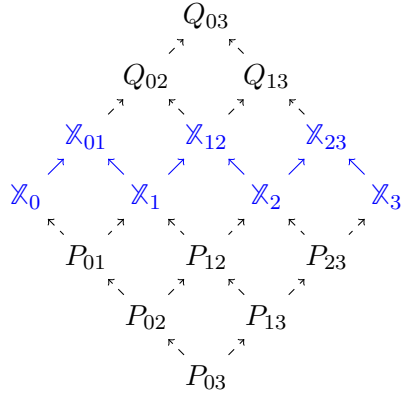
\begin{figure}
\centering
     \begin{tikzpicture}[scale=0.8]
       \node (x0) at (0,0) {\textcolor{blue}{$\X_0$}};
        \node (x1) at (2,0) {\textcolor{blue}{$\X_1$}};
        \node (x2) at (4,0) {\textcolor{blue}{$\X_2$}};
        \node (x3) at (6,0) {\textcolor{blue}{$\X_3$}};
       \node (x01) at (1,1) {\textcolor{blue}{$\X_{01}$}};
        \node (x12) at (3,1) {\textcolor{blue}{$\X_{12}$}};
        \node (x23) at (5,1) {\textcolor{blue}{$\X_{23}$}};
        \draw[->,blue] (x0) edge (x01);
        \draw[->,blue] (x1) edge (x01);
        \draw[->,blue] (x2) edge (x12);
        \draw[->,blue] (x1) edge (x12);
        \draw[->,blue] (x2) edge (x23);
        \draw[->,blue] (x3) edge (x23);
        \node (p01) at (1,-1) {$P_{01}$};
        \node (p12) at (3,-1) {$P_{12}$};
        \node (p23) at (5,-1) {$P_{23}$};
        \draw[<-,dashed] (x0) edge (p01);
        \draw[<-,dashed] (x1) edge (p01);
        \draw[<-,dashed] (x1) edge (p12);
        \draw[<-,dashed] (x2) edge (p12);
        \draw[<-,dashed] (x2) edge (p23);
        \draw[<-,dashed] (x3) edge (p23);

        \node (p02) at (2,-2) {$P_{02}$};
        \node (p13) at (4,-2) {$P_{13}$};
        \draw[->,dashed] (p02) edge (p12);
        \draw[->,dashed] (p02) edge (p01);
        \draw[->,dashed] (p13) edge (p12);
        \draw[->,dashed] (p13) edge (p23);

        \node (p03) at (3,-3) {$P_{03}$};
        \draw[->,dashed] (p03) edge (p13);
        \draw[->,dashed] (p03) edge (p02);

        \node (q02) at (2,2) {$Q_{02}$};
        \node (q13) at (4,2) {$Q_{13}$};
        \draw [->,dashed] (x12) edge (q02);
        \draw [->,dashed] (x01) edge (q02);
        \draw [->,dashed] (x12) edge (q13);
        \draw [->,dashed] (x23) edge (q13);

        \node  (q03) at (3,3) {$Q_{03}$};
        \draw[->,dashed] (q02) edge (q03);
        \draw[->,dashed] (q13) edge (q03);

     \end{tikzpicture}
\caption{The diagram of a normalized zig-zag persistence module completed with the correspondent pullbacks and pushouts.}
\label{figzig}       
\end{figure}     

\noindent The implication operation can now be identified. For instance $(\X_{0}\Rightarrow \X_{1}) = Q_{13}$, $(\X_{1}\Rightarrow \X_{2}) = \X_{23}$ and $(\X_{2}\Rightarrow \X_{3}) = \X_{3}$.
In general, the implication operation between two arbitrary spaces $A=\X_{xy}$ and $B=\X_{zw}$ in a normalized zig-zag module is defined as follows:
 
\begin{theorem}\label{zig}
Let $\HH$ be the underlying Heyting algebra of a normalized zig-zag module and let $\X_k$ and $\X_{ij}$ be arbitrary unrelated vector spaces in the input zig-zag diagram. 
In the related cases we get

\begin{itemize}
\item[] $(\X_i\Rightarrow \X_{ik}) = Q_{0n}$;
\item[] $(\X_i\Rightarrow \X_{ki}) = Q_{0n}$;
\item[] $(\X_{ik}\Rightarrow \X_i) = X_{ik}$;
\item[] $(\X_{ki}\Rightarrow \X_i) = X_{ki}$.
\end{itemize}

If $i\leq j$ then 

\begin{itemize}
\item[] $(\X_i \Rightarrow \X_j) = Q_{jn}$;
\item[] $(\X_j \Rightarrow \X_i) = Q_{0i}$;
\item[] $(\X_{jk} \Rightarrow \X_{ir}) = Q_{0r}$;
\item[] $(\X_{ir} \Rightarrow \X_{jk}) = Q_{jn}$.
\end{itemize}

Note that $Q_{0n}=\X_0$, $Q_{nn}=\X_n$, $Q_{01}=\X_{01}$ and $Q_{(n-1)n}=\X_{(n-1)n}$.
Moreover, $(\X_0\Rightarrow \X_n)=\X_n$ and that $(\X_n\Rightarrow \X_0)=\X_0$
Furthermore, If $i<j$ then

\begin{itemize}
\item[] $(\X_i \Rightarrow \X_{jr})=Q_{jn}$;
\item[] $(\X_j \Rightarrow \X_{ir})=Q_{0r}$;
\item[] $(\X_{jr} \Rightarrow \X_{i})=Q_{0i}$;
\item[] $(\X_{ir} \Rightarrow \X_{j})=Q_{jn}$.
\end{itemize}

This analysis describes all the possible cases.

\end{theorem}

Let us now see how the universal laws of Remark \ref{axm} helps us with further calculus involving the implication operation by simplification of algebraic expressions.
For instance:
\begin{align*}
&(\X_0\Rightarrow (X_2\meet \X_3))=((\X_0\Rightarrow \X_1)\meet (\X_0\Rightarrow \X_3))=(Q_{13}\meet \X_3)=\X_3 \text{   due to (4); }\\
&P_{01}\Rightarrow \X_3=((\X_0\meet \X_1)\Rightarrow \X_3)=(\X_0\Rightarrow (\X_1\Rightarrow \X_3))=(\X_0\Rightarrow \X_3)=\X_3 \text{   due to (9); }\\
&((\X_0\Rightarrow \X_3)\Rightarrow ((\X_1\Rightarrow \X_3)\Rightarrow ((\X_0\join \X_1)\Rightarrow \X_3)))=1=Q_{0n} \text{   due to (12).}
\end{align*}

\subsection{Applications}

In this section, we present the reader with two applications of the theoretic work we developed above. 
Let us first consider the implication operation in standard persistence.
In this case, we are dealing with a totally ordered set of vector spaces and linear maps. 
We can use the implication operation to obtain the most persistent feature, with one arbitrary vector space in the filtration: when taking an element, $A$, of the given input diagram, the computation of $A\Rightarrow A$ gives us the greatest element, $\top$, on the totally ordered set.
When choosing two vector spaces in the filtration we are able to tell their order and compute the element of greatest persistence feature in the following sense: take two vector spaces $A$ and $B$; verify if $B$ is the greatest element of the totally ordered set by analyzing the equality $B=(B\Rightarrow B)$; in the negative case $A\Rightarrow B$ will output 
\begin{itemize}
\item[(1)] the greatest element of the totally ordered set $\top$ whenever $A\leq B$;
\item[(2)] or else it will tell us that $B\leq A$ by outputting $B$.
\end{itemize}
This situation is illustrated in the diagram below.

\[
 \perp \longrightarrow  \ldots \longrightarrow  \X_i \longrightarrow A
 \longrightarrow \X_j \longrightarrow \ldots \longrightarrow  B
 \longrightarrow \X_k \longrightarrow \ldots \longrightarrow \top
\]

\noindent Due to Remark \ref{axm}, $A\Rightarrow B$ already gives us $\top$ while $B\Rightarrow A$ outputs $B$.
We shall now look at this same problem as above, now in the wider context of the underlying Heyting algebra of a bifiltration.
In this case one can fix one parameter and take the filtration correspondent to the other parameter studied in the bifiltration. 
We are thus considering a totally ordered set that is always a sublattice of the underlying Heyting algebra of the given bifiltration. 
The greatest element of such a totally ordered set corresponds to the most persistent feature when one of the parameters is fixed. 
It can be computed using the implication operation as in the standard persistence case, illustrated in Figure \ref{figapp1} (a).
Furthermore, when we study incomparable elements of the bifiltration, we can also detect the implication operation between them, as illustrated in Figure \ref{figapp1} (b).
The output $X$ of the implication $A\Rightarrow B$ is a vector space in the persistence lattice corresponding to the biggest element such that a linear map $(A\meet X) \rightarrow B$ exists.

\begin{figure}
\centering
\subfigure[]{%
\includegraphics[width = 0.3\textwidth,page=1]{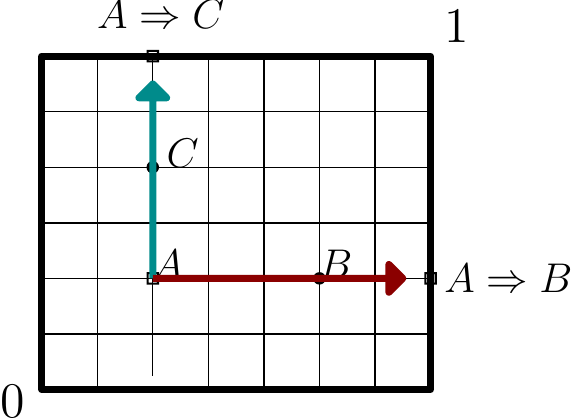}
}     
 \quad
 \subfigure[]{%
 \includegraphics[width = 0.3\textwidth,page=1]{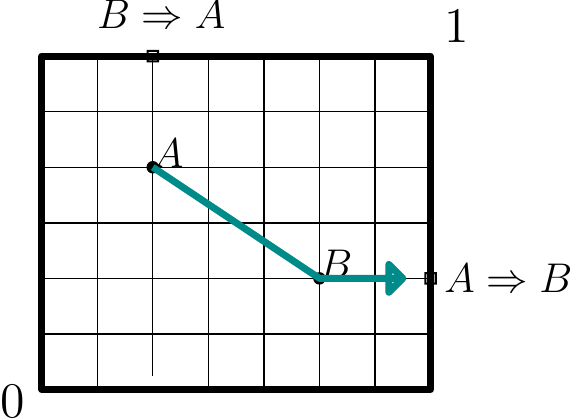}
}
\caption{Interpretation of the implication operation in the framework of persistence when elements of a bifiltration are related (a) and unrelated (b).}
\label{figapp1}       
\end{figure}

Let us look now at the zig-zag persistence case by considering the normalized zig-zag diagram of length 3 given by $\X_{0} \rightarrow \X_{01}\leftarrow \X_{1}\rightarrow \X_{12}\leftarrow \X_{2}\rightarrow \X_{23}\leftarrow \X_{3}$ with the diagram above in subsection \ref{Zig-zag case}.
Let $A$ be $\X_0$ and $B$ be $\X_1$. Then, $A\Rightarrow B$ is $Q_{13}$ which we will call $C$.
Now usually, using the established techniques of zig-zag persistence, we would look at the zig-zag onwards to $\X_{12}$ , $X_2$, $X_{23}$ and $X_3$  to see if we could extend a class which is in both $A$ and $B$ to the other spaces.
With the implication operation, we are able to look at the filtration (as the totally ordered set) from $(A\meet C)$ to $C$.
Recall that $A\meet C=A\meet B$ so that, in this case, the filtration is
\[ (A\meet B) \rto B \rto X_{12} \rto Q_{13}.\]
This sequence contains all the information for bars starting from $A$ and persists to $B$ and onwards.
It is represented in the Hasse diagram of Figure \ref{figapp2}, correspondent to the underlying Heyting algebra including the pullback and pushout constructions denoted by $P_i$ and $Q_i$, respectively. 

\begin{figure}
\centering
     \begin{tikzpicture}[scale=0.8]
       \node (x0) at (0,0) {\textcolor{red}{$A$}};
        \node (x1) at (2,0) {\textcolor{red}{$B$}};
        \node (x2) at (4,0) {$\X_2$};
        \node (x3) at (6,0) {$\X_3$};
       \node (x01) at (1,1) {$\X_{01}$};
        \node (x12) at (3,1) {\textcolor{red}{$\X_{12}$}};
        \node (x23) at (5,1) {$\X_{23}$};
        \draw[->] (x0) edge (x01);
        \draw[->] (x1) edge (x01);
        \draw[->] (x2) edge (x12);
        \draw[->,red, very thick] (x1) edge (x12);
        \draw[->] (x2) edge (x23);
        \draw[->] (x3) edge (x23);
        \node (p01) at (1,-1) {$P_{01}$};
        \node (p12) at (3,-1) {$P_{12}$};
        \node (p23) at (5,-1) {$P_{23}$};
        \draw[<-,dashed] (x0) edge (p01);
        \draw[<-,dashed] (x1) edge (p01);
        \draw[<-,dashed] (x1) edge (p12);
        \draw[<-,dashed] (x2) edge (p12);
        \draw[<-,dashed] (x2) edge (p23);
        \draw[<-,dashed] (x3) edge (p23);

        \node (p02) at (2,-2) {$P_{02}$};
        \node (p13) at (4,-2) {$P_{13}$};
        \draw[->,dashed] (p02) edge (p12);
        \draw[->,dashed] (p02) edge (p01);
        \draw[->,dashed] (p13) edge (p12);
        \draw[->,dashed] (p13) edge (p23);

        \node (p03) at (3,-3) {$P_{03}$};
        \draw[->,dashed] (p03) edge (p13);
        \draw[->,dashed] (p03) edge (p02);

        \node (q02) at (2,2) {$Q_{02}$};
        \node (q13) at (4,2) {\textcolor{red}{$(A\Rightarrow B)$}};
        \draw [->,dashed] (x12) edge (q02);
        \draw [->,dashed] (x01) edge (q02);
        \draw [->,dashed,red,very thick] (x12) edge (q13);
        \draw [->,dashed] (x23) edge (q13);

        \node  (q03) at (3,3) {$Q_{03}$};
        \draw[->,dashed] (q02) edge (q03);
        \draw[->,dashed] (q13) edge (q03);

     \end{tikzpicture}
\caption{Interpretation of the implication operation in the framework of persistence for the zig-zag persistence case.}
\label{figapp2}       
\end{figure}
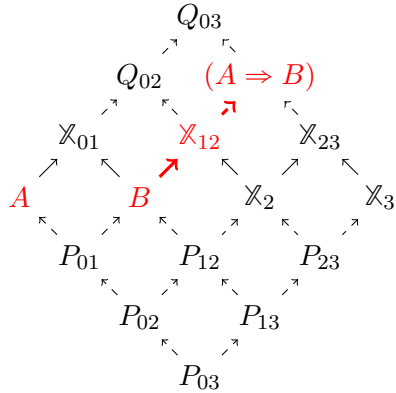

\noindent In general, considering a zig-zag diagram of length $n$. If $A$ is $\X_0$ and $B$ is $\X_1$ we get the following sequence:
\[ (A\meet B) \rto B \rto \X_{12} \rto Q_{13} \rto Q_{14} \rto \ldots \rto Q_{1n}. \]
The part $(A\meet B)$ to $B$ contains complete information about classes which are alive at $A$. From $B$ onwards the information is partial. 
It contains all the classes which potentially persist everywhere and includes some other classes as given (for example consider that the class dies in $X_{12}$, then it is in the cokernel of $X_2 \rto X_{12}$  and so would be in $Q_{13}$). 
Hence, we can look at the implication operation in the zig-zag case as the smallest filtration which contains the classes which persist everywhere.


\section{Discussion}

The internal logic aspects that can be retrieved from the order structure of the underlying algebra of a given diagram of vector spaces and linear maps can provide us of information on persistence.
The contributions of this paper follow the novel ideas of \cite{Skr13} and \cite{SoCG14}. 
With this study we present the reader with some operations that can be defined on standard, zig-zag and multidimensional persistence, and that make sense in connection to formal logic structures that relate to that framework.

While not  a complete study on these matters, in this paper we
highlight new tools which available through lattice structures (cf. \cite{Ba40} and \cite{CatCS}).
These constructive ideas and correspondent constructive logic
represent an alternate viewpoint. 
Due to their computational nature, they can motivate several new
algorithmic applications in which we will focus on in future work.
Moreover, this research represents a step in the clarification of a
topos foundation for the theory of persistence.



\bibliographystyle{plain}



\end{document}